\documentclass[11pt,a4paper]{article}

\usepackage{epsf,epsfig,amsfonts,amsgen,amsmath,amstext,amsbsy,amsopn,amsthm,cases,listings,color
}
\usepackage{ebezier,eepic}
\usepackage{color}
\usepackage{multirow}
\usepackage{epstopdf}
\usepackage{graphicx}
\usepackage{pgf,tikz}
\usepackage{mathrsfs}
\usepackage[marginal]{footmisc}
\usepackage{enumitem}
\usepackage[titletoc]{appendix}
\usepackage{booktabs}
\usepackage{url}
\usepackage{amssymb}
\usepackage{subfigure}
\usepackage{authblk}

\usepackage{pgfplots}

\usepackage{subfigure}
\usepackage{mathrsfs}
\usepackage{xcolor}
\usetikzlibrary{arrows}

\allowdisplaybreaks[1]

\definecolor{uuuuuu}{rgb}{0.27,0.27,0.27}
\definecolor{sqsqsq}{rgb}{0.1255,0.1255,0.1255}

\newtheorem{definition}{Definition} [section]
\newtheorem{theorem}[definition]{Theorem}
\newtheorem{lemma}[definition]{Lemma}
\newtheorem{proposition}[definition]{Proposition}

\newtheorem{problem}[definition]{Problem}
\newtheorem{observation}[definition]{Observation}

\newenvironment{pf}{\noindent{\bf Proof}}

\newcommand{\blow}[2]{#1(\!(#2)\!)}
\newcommand{\multiset}[1]{\{\hspace{-0.25em}\{\hspace{0.1em}#1\hspace{0.1em}\}\hspace{-0.25em}\}}

\newcommand{\hide}[1]{}
\def\multisets#1#2{\ensuremath{\left(\kern-.3em\left(\genfrac{}{}{0pt}{}{#1}{#2}\right)\kern-.3em\right)}}

\begin{document}

\title{\bf\Large Hypergraph Tur\'{a}n densities can have arbitrarily large algebraic degree}

\date{\today}

\author{Xizhi Liu\thanks{Research was supported by ERC Advanced Grant 101020255. Email: xizhi.liu@warwick.ac.uk} }
\author{Oleg Pikhurko\thanks{Research was supported by ERC Advanced Grant 101020255 and
            Leverhulme Research Project Grant RPG-2018-424. Email: o.pikhurko@warwick.ac.uk} }
\affil{Mathematics Institute and DIMAP,\\
            University of Warwick,\\
            Coventry, CV4 7AL, UK}
\maketitle
\begin{abstract}
Grosu [
Journal of Combinatorial Theory (B), \textbf{118} (2016) 137--185]  asked if there exist an integer $r\ge 3$ and a finite family of $r$-graphs whose Tur\'{a}n density, as a real number, has (algebraic) degree greater than~$r-1$. In this note we show that, for all integers $r\ge 3$ and $d$,
there exists a finite family of $r$-graphs whose Tur\'{a}n density has degree at least~$d$, thus answering Grosu's question
in a strong form.
\end{abstract}
\section{Introduction}\label{SEC:Introduction}
For an integer $r\ge 2$, an \emph{$r$-uniform hypergraph} (henceforth, an \emph{$r$-graph}) $H$ is a collection of $r$-subsets of some finite set $V$.
Given a family $\mathcal{F}$ of $r$-graphs, we say $H$ is \emph{$\mathcal{F}$-free}
if it does not contain any member of $\mathcal{F}$ as a subgraph.
The {\em Tur\'{a}n number} $\mathrm{ex}(n,\mathcal{F})$ of $\mathcal{F}$ is the maximum
number of edges in an $\mathcal{F}$-free $r$-graph on $n$ vertices.
The {\em Tur\'{a}n density} $\pi(\mathcal{F} )$ of $\mathcal{F}$ is defined as
$\pi(\mathcal{F}):=\lim_{n\to \infty}\mathrm{ex}(n,\mathcal{F})/{n\choose r}$; the
existence of the limit was established in~\cite{KatonaNemetzSimonovits64}.
The study of $\mathrm{ex}(n,\mathcal{F})$ is one of the central topic in extremal graph and hypergraph theory. For the
hypergraph Tur\'an problem (i.e.\ the case $r\ge 3$), we refer the  reader to the surveys by Keevash~\cite{Keevash11} and Sidorenko~\cite{Sidorenko95}.

For $r\ge 3$, determining the value of $\pi(\mathcal{F})$ for a given $r$-graph family $\mathcal F$ is very difficult in general, and there are only a few known results. For example,  the problem of determining $\pi(K_{\ell}^{r})$ raised by Tur\'{a}n~\cite{TU41} in 1941,
where $K_{\ell}^{r}$ is the complete $r$-graph on $\ell$ vertices, is wide open and the $\$ 500$ prize of
Erd\H{o}s for solving it for at least one pair $\ell>r\ge 3$ is still unclaimed.

\hide{
Indeed, the problem of determining $\pi(K_{\ell}^{r})$ raised by Tur\'{a}n~\cite{TU41},
where $K_{\ell}^{r}$ is the complete $r$-graph on $\ell$ vertices, is still wide open for all $\ell>r\ge 3$.
Erd\H{o}s offered $\$ 500$ for the determination of any $\pi(K_{\ell}^{r})$
with $\ell > r \ge 3$ and $\$ 1000$ for the determination of all $\pi(K_{\ell}^{r})$ with $\ell > r \ge 3$.
The smallest open case $K_{4}^3$ is of particular interest,
for which Tur\'{a}n~\cite{TU41} conjectured that $\pi(K_{4}^{3})  = 5/9$.
Successively better upper bounds for $\pi(K_{4}^{3})$ were obtained by
de Caen $\cite{DC88}$, Giraud (see $\cite{CL99}$), Chung and Lu  $\cite{CL99}$, and Razborov $\cite{RA10}$.
The current record is $\pi(K_{4}^{3}) \le 0.561666$, which was obtained by Razborov  $\cite{RA10}$ using the Flag Algebra machinery.
As for families of $r$-graphs in general, there are only finitely many elements are known to belong to $\Pi_{\mathrm{fin}}^{(r)}$ before 2006 even for $r=3$
(see e.g. \cite{Furedi91,DF00,MR02,FPS03,FS05,KS05a,KS05b,FPS05}).
The first infinite sequence of elements that belong to $\Pi_{\mathrm{fin}}^{(r)}$ was found by Mubayi in~\cite{MU06}, where he proved that
\begin{align*}
\left\{\frac{(\ell)_{r}}{\ell^r} \colon \ell \in \mathbb{N} \text{ and } \ell\ge r \right\}
\subset \Pi_{\mathrm{fin}}^{(r)} \quad\text{for every } r \ge 3.
\end{align*}
Later, more elements were proved to belong to $\Pi_{\mathrm{fin}}^{(r)}$,
and we refer the reader to \cite{KE11,BIJ17,NY17,NY18,JPW18,BNY19,YP22} for more recent research.
}

For every integer $r\ge 2$, define
\begin{align*}
\Pi_{\mathrm{fin}}^{(r)} & := \left\{\pi(\mathcal{F}) \colon \text{$\mathcal{F}$ is a finite family of $r$-graphs} \right\}, \quad\text{and} \\
\Pi_{\infty}^{(r)} & := \left\{\pi(\mathcal{F}) \colon \text{$\mathcal{F}$ is a (possibly infinite) family of $r$-graphs} \right\}.
\end{align*}

For $r = 2$ the celebrated Erd\H{o}s--Stone--Simonovits theorem~\cite{ErdosSimonovits66,ErdosStone46} determines the Tur\'{a}n density for every family $\mathcal{F}$ of graphs; in particular, it holds that
\begin{align*}
\Pi_{\infty}^{(2)} = \Pi_{\mathrm{fin}}^{(r)}=  \{1\}\cup \left\{1-{1}/{k} \colon \mbox{integer $k\ge 1$}\right\}.
\end{align*}

The problem of understanding the sets $\Pi_{\mathrm{fin}}^{(r)}$ and $\Pi_{\infty}^{(r)}$ of possible $r$-graph Tur\'an densities for $r\ge 3$
has attracted a lot of attention. One of the earliest results here is the theorem of Erd\H{o}s~\cite{E64} from the 1960s that $\Pi_{\infty}^{(r)} \cap (0, r!/r^r) = \emptyset$ for every integer $r\ge 3$. However, our understanding of the locations and the lengths of other maximal intervals avoiding $r$-graph Tur\'an densities and the right accumulation points of $\Pi_{\infty}^{(r)}$ (the so-called \emph{jump problem}) is very limited; for some results in this direction see e.g.~\cite{BaberTalbot11,FranklPengRodlTalbot07,FranklRodl84,Pikhurko15,YanPeng21}.

It is known that the set $\Pi_{\infty}^{(r)}$ is the topological closure of $\Pi_{\mathrm{fin}}^{(r)}$ (and thus the former set is easier to understand) and that $\Pi_{\infty}^{(r)}$ has cardinality of continuum (and thus is strictly larger than the countable set $\Pi_{\mathrm{fin}}^{(r)}$), see respectively Proposition~1 and Theorem~2 in~\cite{PI14}.

For a while it was open whether $\Pi_{\mathrm{fin}}^{(r)}$ can contain an irrational number (see the conjecture of Chung and Graham in~\cite[Page 95]{CG98}), until such examples were independently found by  Baber and Talbot~\cite{BT12} and by the second author~\cite{PI14}. However, the question of Jacob Fox (\cite[Question~27]{PI14}) whether $\Pi_{\mathrm{fin}}^{(r)}$ can contain a transcendental number remains open.

Grosu~\cite{Grosu16} initiated a systematic study of algebraic properties of the sets $\Pi_{\mathrm{fin}}^{(r)}$ and $\Pi_{\infty}^{(r)}$. He proved a number of general results that, in particular, directly give further examples of irrational Tur\'an densities.

Recall that the \emph{(algebraic) degree} of a real number $\alpha$ is the minimum degree of a non-zero polynomial $p$ with integer coefficients that vanishes on~$\alpha$; it is defined to be $\infty$ if no such $p$ exists (that is, if the real $\alpha$ is transcendental). In the same paper, Grosu~\cite[Problem~3]{Grosu16} posed the following question.

\begin{problem}[Grosu]\label{PROB:Grosu}
Does there exist an integer $r\ge 3$ such that $\Pi_{\mathrm{fin}}^{(r)}$ contains an algebraic number $\alpha$ of degree strictly larger than~$r-1$?
\end{problem}

Apparently,  all $r$-graph Tur\'an densities that Grosu knew or could produce with his machinery had degree at most~$r-1$, explaining this expression in his question.
His motivation for asking this question was that if, on input $\mathcal F$, we can compute
an upper bound on the degree of $\pi(\mathcal F)$ as well as on the absolute values of the coefficients of its minimal polynomial, then we can compute $\pi(\mathcal F)$ exactly, see the discussion in~\cite[Page 140]{Grosu16}.

In this short note we answer Grosu's question in the following stronger form.

\begin{theorem}\label{THM:min-poly-degree}
For every integer $r\ge 3$ and for every integer $d$ there exists an algebraic number in $\Pi_{\mathrm{fin}}^{(r)}$
whose minimal polynomial has degree at least $d$.
\end{theorem}

Our proof for Theorem~\ref{THM:min-poly-degree} is constructive; in particular,
for $r=3$ we will show that the following infinite sequence is contained in $\Pi_{\mathrm{fin}}^{(3)}$:
\begin{align}\label{eq:Seq}
\frac{1}{\sqrt{3}},\quad
\frac{1}{\sqrt{3-\frac{2}{\sqrt{3}}}},\quad
\frac{1}{\sqrt{3-\frac{2}{\sqrt{3-\frac{2}{\sqrt{3}}}}}},\quad
\frac{1}{\sqrt{3-\frac{2}{\sqrt{3-\frac{2}{\sqrt{3-\frac{2}{\sqrt{3}}}}}}}},\quad
\ldots\quad .
\end{align}


\section{Preliminaries}\label{SEC:Prelim}
In this section, we introduce some preliminary definitions and results that will be used later.

For an integer $r\ge 2$,
an ($r$-uniform) \emph{pattern} is a pair $P=(m,\mathcal{E})$, where $m$ is a positive integer,
$\mathcal{E}$ is a collection of $r$-multisets
on $[m]:=\{1,\ldots,m\}$, where by an \emph{$r$-multiset} we mean an unordered collection of $r$ elements with repetitions allowed.
Let $V_1,\dots,V_m$
be disjoint sets and let $V=V_1\cup\dots\cup V_m$.
The \emph{profile} of an $r$-set $R\subseteq V$ (with respect to $V_1,\dots,V_m$) is
the $r$-multiset on $[m]$ that contains element $i$ with multiplicity $|R \cap V_i|$ for every $i\in [m]$.
For an $r$-multiset $S \subseteq [m]$, let
$\blow{S}{V_1,\dots,V_m}$ consist of all $r$-subsets of $V$ whose profile is $S$.
We call this $r$-graph the \emph{blowup} of $S$ and
the $r$-graph
$$
\blow{\mathcal{E}}{V_1,\dots,V_m} := \bigcup_{S\in \mathcal{E}} \blow{S}{V_1,\dots,V_m}
$$
is called the \emph{blowup} of $\mathcal{E}$ (with respect to $V_1,\dots,V_m$).
We say that an $r$-graph $H$ is a \emph{$P$-construction} if it is a blowup of $\mathcal{E}$. Note that these are special cases of the more general definitions from~\cite{PI14}.

It is easy to see that the notion of a pattern is a generalization of a hypergraph,
since every $r$-graph is a pattern in which $\mathcal{E}$ is a collection of (ordinary) $r$-sets.
For most families $\mathcal{F}$ whose Tur\'an problem was resolved, the extremal $\mathcal{F}$-free constructions are blowups of some simple pattern.
For example, let $P_{B} := (2, \left\{ \multiset{1,2,2}, \multiset{1,1,2} \right\})$, where we use $\multiset{\empty}$ to distinguish multisets from ordinary sets.
Then a $P_{B}$-construction is a $3$-graph $H$ whose vertex set can be partitioned into two parts $V_1$ and $V_2$
such that $H$ consists of all triples that have nonempty intersections with both $V_1$ and $V_2$.
A famous result in the hypergraph Tur\'{a}n theory is that the pattern $P_{B}$
characterizes the structure of all maximum 3-graphs of sufficiently large order that do not contain a Fano plane (see \cite{DF00,FS05,KS05b}).

For a pattern $P = (m, \mathcal{E})$, let the \emph{Lagrange polynomial} of $\mathcal{E}$ be
\begin{align*}
\lambda_{\mathcal{E}}(x_1,\dots,x_m)
:= r!\,\sum_{E\in \mathcal{E}}\; \prod_{i=1}^m\; \frac{x_i^{E(i)}}{E(i)!},
\end{align*}
where $E(i)$ is the multiplicity of $i$ in the $r$-multiset $E$.
In other words, $\lambda_{\mathcal{E}}$ gives
the asymptotic edge density of a large blowup of $\mathcal{E}$, given its relative part sizes~$x_i$.

The \emph{Lagrangian} of $P$ is defined as follows:
\begin{align*}
\lambda(P) := \sup\left\{\lambda_{\mathcal{E}}(x_1,\dots,x_m) \colon (x_1,\dots,x_m)\in \Delta_{m-1} \right\},
\end{align*}
where $\Delta_{m-1}:=\{(x_1,\dots,x_m)\in [0,1]^m\colon x_1+\ldots+x_m=1\}$ is the standard $(m-1)$-dimensional simplex in $\mathbb{R}^m$. Since we maximise a polynomial (a continuous function) on a compact space, the supremum is in fact the maximum and we call the vectors in $\Delta_{m-1}$ attaining it \emph{$P$-optimal}.
Note that the Lagrangian of a pattern is a generalization
of the well-known \emph{hypergraph Lagrangian} that
has been successfully applied to Tur\'an-type problems (see e.g.~\cite{BaberTalbot11,FranklRodl84,YP22}), with
the basic idea going back to Motzkin and Straus~\cite{MS65}.

For $i\in [m]$ let $P-i$ be the pattern obtained
from $P$ by \emph{removing index $i$}, that is, we remove $i$ from $[m]$ and delete
all multisets containing $i$ from $E$ (and relabel the remaining indices to form the set $[m-1]$).
We call $P$ \emph{minimal} if $\lambda(P-i)$ is strictly smaller than $\lambda(P)$ for every $i\in [m]$, or equivalently if no $P$-optimal vector has a zero entry.
For example, the 2-graph pattern $P:= (3,\{\,\multiset{1,2},\multiset{1,3}\,\})$ is not minimal as $\lambda(P)=\lambda(P-3)=1/2$.

In~\cite{PI14}, the second author studied the relations between possible hyrergraph Tur\'{a}n densities and patterns. One of the main results from~\cite{PI14} is as follows.

\begin{theorem}[\cite{PI14}]\label{THM:min-pattern}
For every minimal pattern $P$ there exists a finite family $\mathcal{F}$ of $r$-graphs
such that $\pi(\mathcal{F}) = \lambda(P)$,
and moreover, every maximum $\mathcal{F}$-free $r$-graph is a $P$-construction.
\end{theorem}

Let $r \ge 3$ and $s \ge 1$ be two integers.
Given an $r$-uniform pattern $P = (m,\mathcal{E})$, one can create an $(r+s)$-uniform pattern $P+s := (m+s, \hat{\mathcal{E}})$ in the following way:
for every $E\in \mathcal{E}$ we insert the $s$-set $\{m+1,\ldots,m+s\}$ into $E$, and let $\hat{\mathcal{E}}$ denote the resulting family of $(r+s)$-multisets.
For example, if $P = (3, \left\{ \multiset{1,2,3}, \multiset{1,3,3}, \multiset{2,3,3} \right\})$, then
$P+1 = (4, \left\{ \multiset{1,2,3,4}, \multiset{1,3,3,4}, \multiset{2,3,3,4} \right\})$.

The following observation follows easily from the definitions.

\begin{observation}\label{OBS:plus-s-minimal}
If $P$ is a minimal pattern, then $P+s$ is a minimal pattern for every integer $s\ge 1$.
\end{observation}

For the Lagrangian of $P+s$ we have the following result.

\begin{proposition}\label{PROP:Lagrangina-plus-s}
Suppose that $r\ge 2$ is an integer and $P$ is an $r$-uniform pattern.
Then for every integer $s\ge 1$ we have
$$
\lambda(P+s) =\frac{r^r (s+r)!}{(r+s)^{r+s} r!}\, \lambda(P).
$$
In particular,  the real numbers $\lambda(P+s)$ and $\lambda(P)$ have the same degree.
\end{proposition}
\begin{proof}
Assume that $P = (m, \mathcal{E})$.
Let $\hat{P} := P+s = (m+s, \widehat{\mathcal{E}})$.
Let $(x_1, \ldots, x_{m+s}) \in \Delta_{m+s-1}$ be a $\hat{P}$-optimal vector.
Note from the definition of Lagrange polynomial that
$$\lambda(\hat{P})
= \lambda_{\hat{\mathcal{E}}}(x_1, \ldots, x_{m+s})
= \frac{(r+s)!}{r!}\,\lambda_{\mathcal{E}}(x_1, \ldots, x_{m})  \prod_{i= m+1}^{m+s}x_i.
$$
Let $x := \frac 1s\, \sum_{i= m+1}^{m+s}x_i$ and note that $\sum_{i= 1}^{m}x_i = 1-sx$.
Since $\lambda_{{\mathcal{E}}}$ is a homogenous polynomial of degree $r$, we have
\begin{align*}
\lambda_{{\mathcal{E}}}(x_1, \ldots, x_{m})
= \lambda_{\mathcal{E}}\left(\frac{x_1}{1-sx}, \ldots, \frac{x_m}{1-sx} \right) (1-sx)^r
\le \lambda(P) (1-sx)^r.
\end{align*}
This and the AM-GM inequality give that
\begin{align*}
\lambda(\hat{P})
= \frac{(r+s)!}{r!}\, \lambda_{\mathcal{E}}(x_1, \ldots, x_{m})  \prod_{i= m+1}^{m+s}x_i
\le \frac{(r+s)!}{r!}\, \lambda(P) (1-sx)^r x^s.
\end{align*}
For $x\in [0,1/s]$,  the function $(1-sx)^r (rx)^s$, as the product of $s+r$ non-negative terms summing to $r$, is maximized when all terms are equal, that is, at $x = \frac{1}{r+s}$.
So
\begin{align*}
\lambda(\hat{P})
\le \frac{(r+s)!}{r!}\, \lambda(P) (1-sx)^r x^s
\le \frac{r^r (s+r)!}{(r+s)^{r+s} r!}\, \lambda(P).
\end{align*}
To prove the other direction of this inequality, observe that if
we take $(x_1, \ldots, x_{m}) = \frac{r}{r+s}\,(y_1, \ldots, y_m)$, where $(y_1, \ldots, y_m) \in \Delta_{m-1}$ is $P$-optimal, and take $x_{m+1} = \cdots = x_{m+s} = \frac{1}{r+s}$,
then all inequalities above hold with equalities.
\end{proof}

\section{Proof of Theorem~\ref{THM:min-poly-degree}}\label{SEC:Proof}
In this section we prove Theorem~\ref{THM:min-poly-degree}.
By Theorem~\ref{THM:min-pattern}, it suffices to find a sequence of $r$-uniform minimal patterns $(P_k)_{k=1}^{\infty}$
such that the degree of the real number $\lambda(P_k)$ goes to infinity as $k$ goes to infinity.
Furthermore, by Observation~\ref{OBS:plus-s-minimal} and Proposition~\ref{PROP:Lagrangina-plus-s}, it suffices to find such a sequence for $r=3$.
So we will assume that $r=3$ in the rest of this note.
%
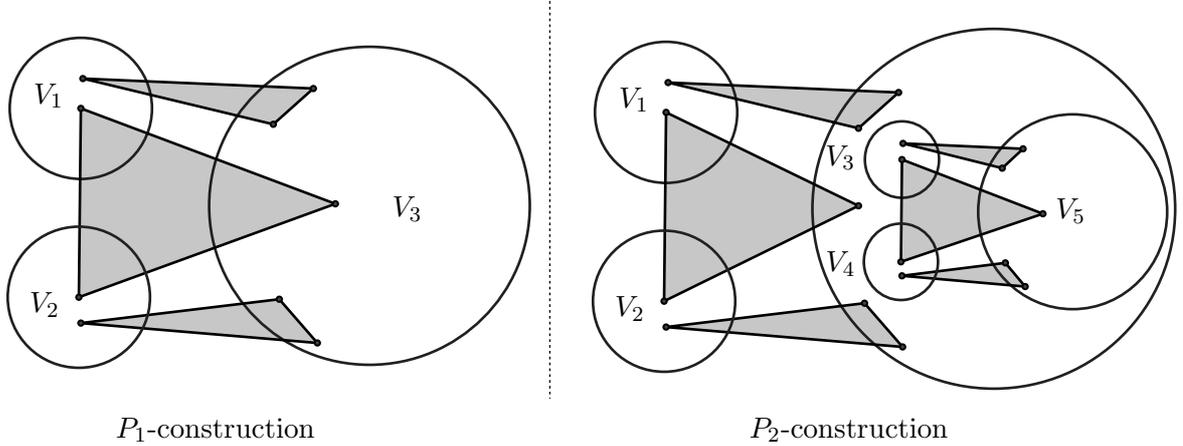
\begin{figure}[htbp]
\centering
\tikzset{every picture/.style={line width=0.75pt}} 
\begin{tikzpicture}[x=0.75pt,y=0.75pt,yscale=-1,xscale=1]
\draw[line width=1pt,color=sqsqsq]   (25.5,69) .. controls (25.5,49.39) and (41.39,33.5) .. (61,33.5) .. controls (80.61,33.5) and (96.5,49.39) .. (96.5,69) .. controls (96.5,88.61) and (80.61,104.5) .. (61,104.5) .. controls (41.39,104.5) and (25.5,88.61) .. (25.5,69) -- cycle ;
\draw[line width=1pt,color=sqsqsq]   (125,118) .. controls (125,73.82) and (160.82,38) .. (205,38) .. controls (249.18,38) and (285,73.82) .. (285,118) .. controls (285,162.18) and (249.18,198) .. (205,198) .. controls (160.82,198) and (125,162.18) .. (125,118) -- cycle ;
\draw[line width=1pt, fill=sqsqsq,fill opacity=0.25]  (177,59) -- (62,54) -- (157,77) -- (177,59) ;
\draw[line width=1pt, fill=sqsqsq,fill opacity=0.25]   (61,177) -- (160,165) -- (179,187) --(61,177) ;
\draw[line width=1pt, fill=sqsqsq,fill opacity=0.25]   (188,117) -- (61,69) -- (60,164) -- (188,117) ;
\draw[line width=1pt,color=sqsqsq]   (24.5,164) .. controls (24.5,144.39) and (40.39,128.5) .. (60,128.5) .. controls (79.61,128.5) and (95.5,144.39) .. (95.5,164) .. controls (95.5,183.61) and (79.61,199.5) .. (60,199.5) .. controls (40.39,199.5) and (24.5,183.61) .. (24.5,164) -- cycle ;
\draw[line width=1pt,color=sqsqsq]   (317.5,71) .. controls (317.5,51.39) and (333.39,35.5) .. (353,35.5) .. controls (372.61,35.5) and (388.5,51.39) .. (388.5,71) .. controls (388.5,90.61) and (372.61,106.5) .. (353,106.5) .. controls (333.39,106.5) and (317.5,90.61) .. (317.5,71) -- cycle ;
\draw[line width=1pt,color=sqsqsq]   (426,119.5) .. controls (426,69.52) and (466.52,29) .. (516.5,29) .. controls (566.48,29) and (607,69.52) .. (607,119.5) .. controls (607,169.48) and (566.48,210) .. (516.5,210) .. controls (466.52,210) and (426,169.48) .. (426,119.5) -- cycle ;
\draw[line width=1pt,fill=sqsqsq,fill opacity=0.25]   (469,61) -- (354,56) -- (449,79) -- (469,61) ;
\draw[line width=1pt,fill=sqsqsq,fill opacity=0.25]   (353,179) -- (452,167) -- (471,189) -- (353,179);
\draw[line width=1pt,fill=sqsqsq,fill opacity=0.25]   (449,118) -- (353,71) -- (352,166) -- (449,118) ;
\draw[line width=1pt,color=sqsqsq]   (316.5,166) .. controls (316.5,146.39) and (332.39,130.5) .. (352,130.5) .. controls (371.61,130.5) and (387.5,146.39) .. (387.5,166) .. controls (387.5,185.61) and (371.61,201.5) .. (352,201.5) .. controls (332.39,201.5) and (316.5,185.61) .. (316.5,166) -- cycle ;
\draw[line width=1pt,color=sqsqsq]   (452.19,94.74) .. controls (452.19,84.12) and (460.47,75.52) .. (470.69,75.52) .. controls (480.9,75.52) and (489.18,84.12) .. (489.18,94.74) .. controls (489.18,105.36) and (480.9,113.96) .. (470.69,113.96) .. controls (460.47,113.96) and (452.19,105.36) .. (452.19,94.74) -- cycle ;
\draw[line width=1pt,color=sqsqsq]   (508.71,121) .. controls (508.71,93.94) and (529.82,72) .. (555.86,72) .. controls (581.89,72) and (603,93.94) .. (603,121) .. controls (603,148.06) and (581.89,170) .. (555.86,170) .. controls (529.82,170) and (508.71,148.06) .. (508.71,121) -- cycle ;
\draw[line width=1pt, fill=sqsqsq,fill opacity=0.25]   (471.21,86.62) -- (520.7,99.07) -- (531.11,89.33)--(471.21,86.62) ;
\draw[line width=1pt, fill=sqsqsq,fill opacity=0.25]    (470.69,153.22) -- (522.26,146.72) -- (532.16,158.63) -- (470.69,153.22);
\draw[line width=1pt, fill=sqsqsq,fill opacity=0.25]   (470.69,94.74) -- (470.17,146.18) -- (541,122)--(470.69,94.74) ;
\draw[line width=1pt,color=sqsqsq]   (451.67,146.18) .. controls (451.67,135.56) and (459.95,126.96) .. (470.17,126.96) .. controls (480.38,126.96) and (488.66,135.56) .. (488.66,146.18) .. controls (488.66,156.79) and (480.38,165.4) .. (470.17,165.4) .. controls (459.95,165.4) and (451.67,156.79) .. (451.67,146.18) -- cycle ;
\draw[line width=0.5pt,dash pattern=on 1pt off 1.2pt]    (295,15) -- (295,215) ;

\draw (77,223) node [anchor=north west][inner sep=0.75pt]   [align=left] {$P_1$-construction};
\draw (393,223) node [anchor=north west][inner sep=0.75pt]   [align=left] {$P_2$-construction};
\draw (36,55) node [anchor=north west][inner sep=0.75pt]   [align=left] {$V_1$};
\draw (34,160) node [anchor=north west][inner sep=0.75pt]   [align=left] {$V_2$};
\draw (215,112) node [anchor=north west][inner sep=0.75pt]   [align=left] {$V_3$};
\draw (328,57) node [anchor=north west][inner sep=0.75pt]   [align=left] {$V_1$};
\draw (326,162) node [anchor=north west][inner sep=0.75pt]   [align=left] {$V_2$};
\draw (431,86) node [anchor=north west][inner sep=0.75pt]   [align=left] {$V_3$};
\draw (431.17,139.4) node [anchor=north west][inner sep=0.75pt]   [align=left] {$V_4$};
\draw (546,113) node [anchor=north west][inner sep=0.75pt]   [align=left] {$V_5$};
\draw [fill=uuuuuu] (177,59) circle (1pt);
\draw [fill=uuuuuu] (62,54) circle (1pt);
\draw [fill=uuuuuu] (157,77) circle (1pt);

\draw [fill=uuuuuu] (179,187) circle (1pt);
\draw [fill=uuuuuu] (61,177) circle (1pt);
\draw [fill=uuuuuu] (160,165) circle (1pt);

\draw [fill=uuuuuu] (188,117) circle (1pt);
\draw [fill=uuuuuu] (61,69) circle (1pt);
\draw [fill=uuuuuu] (60,164) circle (1pt);
\draw [fill=uuuuuu] (469,61) circle (1pt);
\draw [fill=uuuuuu] (354,56) circle (1pt);
\draw [fill=uuuuuu] (449,79) circle (1pt);
\draw [fill=uuuuuu] (471,189) circle (1pt);
\draw [fill=uuuuuu] (353,179) circle (1pt);
\draw [fill=uuuuuu] (452,167) circle (1pt);
\draw [fill=uuuuuu] (449,118) circle (1pt);
\draw [fill=uuuuuu] (353,71) circle (1pt);
\draw [fill=uuuuuu] (352,166) circle (1pt);
\draw [fill=uuuuuu] (531.11,89.33) circle (1pt);
\draw [fill=uuuuuu] (471.21,86.62) circle (1pt);
\draw [fill=uuuuuu] (520.7,99.07) circle (1pt);
\draw [fill=uuuuuu] (532.16,158.63) circle (1pt);
\draw [fill=uuuuuu] (470.69,153.22) circle (1pt);
\draw [fill=uuuuuu] (522.26,146.72) circle (1pt);
\draw [fill=uuuuuu] (541,122) circle (1pt);
\draw [fill=uuuuuu] (470.69,94.74) circle (1pt);
\draw [fill=uuuuuu] (470.17,146.18) circle (1pt);
\end{tikzpicture}
\caption{Constructions with one level and two levels.}
\label{fig:P1-P2}
\end{figure}

To start with, we let $P_1 := \left(3, \left\{ \multiset{1,2,3},\multiset{1,3,3}, \multiset{2,3,3} \right\} \right)$.
Recall that a $3$-graph $H$ is a $P_{1}$-construction (see Figure~\ref{fig:P1-P2}) if there exists a partition $V(H) = V_1 \cup V_2 \cup V_3$
such that the edge set of $H$ consists of
\begin{enumerate}[label=(\alph*)]
\item all triples that have one vertex in each $V_i$,
\item all triples that have one vertex in $V_1$ and two vertices in $V_3$, and
\item all triples that have one vertex in $V_2$ and two vertices in $V_3$.
\end{enumerate}
The pattern $P_1$ was studied by Yan and Peng in~\cite{YP22}, where they proved that  there exists a single $3$-graph
whose Tur\'an density is given by $P_1$-constructions which, by $\lambda(P_1) = 1/\sqrt{3}$, is an irrational number. It seems that some other patterns could be used
 to prove Theorem~\ref{THM:min-poly-degree}; however, the obtained sequence of Tur\'an densities (i.e.\ the sequence in~\eqref{eq:Seq})
produced by using $P_1$ is nicer than those produced by the other patterns that we tried.

Next, we define the pattern $P_{k+1}=(2k+3,\mathcal{E}_{k+1})$ for every $k\ge 1$ inductively.
It is easier to define what a $P_{k+1}$-construction is rather than to write down the definition of $P_{k+1}$:
for every integer $k\ge 1$ a $3$-graph $H$ is a \emph{$P_{k+1}$-construction} if there exists a partition $V(H) = V_1 \cup V_2 \cup V_3$
such that
\begin{enumerate}[label=(\alph*)]
\item the induced subgraph $H[V_3]$ is a $P_k$-construction, and
\item $H\setminus H[V_3]$ consists of all triples whose profile is in
$\left\{ \multiset{1,2,3},\multiset{1,3,3}, \multiset{2,3,3} \right\}$. 
\end{enumerate}

The pattern $P_{k}$ can be written down explicitly, although this is not necessary for our proof later.
For example, $P_2 = (5, \mathcal{E}_2)$ (see Figure~\ref{fig:P1-P2}), where
\begin{align*}
\mathcal{E}_2 =&  \left\{
\multiset{1,2,3},
\multiset{1,2,4},
\multiset{1,2,5},
\multiset{1,3,3},
\multiset{1,3,4},
\multiset{1,3,5}, \right. \\ & \quad
\multiset{1,4,4},
\multiset{1,4,5},
\multiset{1,5,5},
\multiset{2,3,3},
\multiset{2,3,4},
\multiset{2,3,5},\\ & \quad   \left.
\multiset{2,4,4},
\multiset{2,4,5},
\multiset{2,5,5},
\multiset{3,4,5},
\multiset{3,5,5}, \multiset{4,5,5}
\right\}.
\end{align*}

Our first result determines the Lagrangian of $P_k$ for every $k\ge 1$.
For convenience, we set $P_0 := (1, \{\emptyset\})$ and $\lambda_0 := 0$.

\begin{proposition}\label{PROP:Lagrangian-Pk}
For every integer $k\ge 0$, we have $\lambda(P_{k+1}) = {1}/{\sqrt{3-2\lambda(P_k)}}$ and  the pattern $P_{k+1}$ is minimal.
In particular, $\left(\lambda(P_{k})\right)_{k=1}^{\infty}$ is the sequence in~\eqref{eq:Seq}.
\end{proposition}
\begin{proof}
We use induction on $k$ where the base $k=0$ is easy to check directly (or can be derived by adapting the forthcoming induction step to work for $k=0$). Let $k\ge 1$.

Let us prove that $\lambda(P_{k+1}) = {1}/{\sqrt{3-2\lambda(P_k)}}$.
Recall that $P_k = (2k+1, \mathcal{E}_k)$ and $P_{k+1} = (2k+3, \mathcal{E}_{k+1})$.
Let $(x_1, \ldots, x_{2k+3}) \in \Delta_{2k+2}$ be a $P_{k+1}$-optimal vector.
Let $x := \sum_{i=3}^{2k+3} x_i= 1-x_1-x_2$.
It follows from the definitions of $P_{k+1}$ and the Lagrange polynomial that
\begin{align}\label{equ:lambda-P-k+1}
\lambda(P_{k+1})
= \lambda_{\mathcal{E}_{k+1}}(x_1, \ldots, x_{2k+3})
= 6\left(x_1x_2x+(x_1+x_2)\frac{x^2}{2}\right) + \lambda_{\mathcal{E}_k}(x_3, \ldots, x_{2k+3}).
\end{align}
Since $\lambda_{\mathcal{E}_k}(x_3, \ldots, x_{2k+3})$ is a homogeneous polynomial of degree $3$, we have
\begin{align*}
\lambda_{\mathcal{E}_k}(x_3, \ldots, x_{2k+3})
= \lambda_{\mathcal{E}_k}\left(\frac{x_3}{x}, \ldots, \frac{x_{2k+3}}{x}\right)x^3
\le \lambda(P_k) x^3.
\end{align*}
So it follows from (\ref{equ:lambda-P-k+1}) and the 2-variable AM-GM inequality that
\begin{align*}
\lambda(P_{k+1})
& \le 6\left( \left(\frac{x_1+x_2}{2}\right)^2x+ (x_1+x_2)\frac{x^2}{2} \right)+ \lambda(P_k) x^3 \\
& = 6\left( \left(\frac{1-x}{2}\right)^2x+ (1-x)\frac{x^2}{2} \right)+ \lambda(P_k) x^3
 = \frac{3 x - \left(3-2\lambda(P_k)\right)x^3}{2}.
\end{align*}
Since $0\le \lambda(P_k)\le 1$,  one can easily show by taking the derivative that the maximum of the function $\left(3 x - \left(3-2\lambda(P_k)\right)x^3\right)/2$
on $[0,1]$ is achieved if and only if $x = {1}/{\sqrt{3-2\lambda(P_k)}}$,
and the maximum value is ${1}/{\sqrt{3-2\lambda(P_k)}}$.
This proves that $\lambda(P_{k+1}) \le {1}/{\sqrt{3-2\lambda(P_k)}}$.

To prove the other direction of this inequality,
one just need to observe that when we choose
\begin{align}\label{equ:maximizer-Pk}
x_1 = x_2 = \frac{1}{2} - \frac{1}{2\sqrt{3-2\lambda(P_k)}}\quad\text{and}\quad
(x_3, \ldots, x_{2k+3}) = \frac{1}{\sqrt{3-2\lambda(P_k)}}\,(y_1, \ldots, y_{2k+1})
\end{align}
where $(y_1, \ldots, y_{2k+1}) \in \Delta_{2k}$ is a $P_k$-optimal vector,
then all inequalities above hold with equality.
Therefore, $\lambda(P_{k+1}) = {1}/{\sqrt{3-2\lambda(P_k)}}$.

To prove that $P_{k+1}$ is minimal, take any $P_{k+1}$-optimal vector $(x_1, \ldots, x_{2k+3}) \in \Delta_{2k+2}$; we have to show that it has no zero entries. This vector attains equality in all our inequalities above, which routinely implies that $(x_1, \ldots, x_{2k+3})$ must satisfy (\ref{equ:maximizer-Pk}), for some $P_k$-optimal vector $(y_1,\ldots,y_{2k+1})$.
We see that $x_1 = x_2$ are both non-zero because the sequence $(\lambda(P_0),\ldots,\lambda(P_{k+1}))$ is strictly increasing (since $x<1/\sqrt{3-2x}$ for all $x\in [0,1)$)
and thus $\lambda(P_k)<1$.
The remaining conclusion that  $x_3, \ldots, x_{2k+3}$ are non-zero follows from the induction hypothesis on $(y_1, \ldots, y_{2k+1})$.
\end{proof}

In order to finish the proof of Theorem~\ref{THM:min-poly-degree} it suffices to prove that the degree of $\mu_k:=\lambda(P_k)$
goes to infinity as $k\to\infty$. This is achieved by the last claim of the following lemma.

\begin{lemma}\label{lm:new} Let $p_1(x):=3x^2-1$ and inductively for $k=1,2,\dots$ define 
$$
 p_{k+1}(x)=(2x^2)^{2^k} p_k\left(\frac{3x^2-1}{2x^2}\right),\quad \mbox{for $x\in\mathbb R$}.
$$
 Then the following claims hold for each $k\in\mathbb N \colon$ 
\begin{enumerate}[label=(\alph*)]
\item\label{it:n1} $p_k(\mu_k)=0$; 
\item\label{it:n2} $p_k$ is a polynomial of degree at most $2^k$ with integer coefficients: $p_k(x)=\sum_{i=0}^{2^k} c_{k,i} x^i$ for some $c_{k,i}\in\mathbb Z$;
\item\label{it:n3} 
 the integers $b_{k,i}:=c_{k,i}$ for even $k$ and $b_{k,i}:=c_{k,2^k-i}$ for odd $k$ satisfy the following:
 \begin{enumerate}[label=(c.\roman*)]
 \item\label{it:s1} for each integer $i$ with $0\le i\le 2^k$, $3$ divides $b_{k,i}$ if and only if $i\not=2^k$;
 \item\label{it:s2} $9$ does not divide $b_{k,0}$;
\end{enumerate}
\item\label{it:n4} the polynomial $p_k$ is irreducible of degree exactly $2^k$;
\item\label{it:n5} the degree of $\mu_k$ is $2^k$.
 \end{enumerate}
\end{lemma}

\begin{proof} Let us use induction on $k$. All stated claims are clearly satisfied for $k=1$, when $p_1(x)=3x^2-1$ and $\mu_1=1/\sqrt3$. Let us prove them for $k+1$ assuming that they hold for some $k\ge 1$.

For Part~\ref{it:n1}, we have by Proposition~\ref{PROP:Lagrangian-Pk} that
$$
 \frac{3\mu_{k+1}^2-1}{2\mu_{k+1}^2}= \frac{3/(3-2\mu_k)-1}{2/(3-2\mu_k)}=\mu_k
$$ 
 and thus $p_{k+1}(\mu_{k+1})=(2\mu_{k+1}^2)^{2^k} p_k(\mu_k)$, which is 0 by induction.

Part~\ref{it:n2}
also follows easily from the induction assumption:
 \begin{equation}\label{eq:p}
p_{k+1}(x)=(2x^2)^{2^k}\sum_{i=0}^{2^k} c_{k,i} \left(\frac{3x^2-1}{2x^2}\right)^i
 = \sum_{i=0}^{2^k} c_{k,i}(3x^2-1)^i (2x^2)^{2^k-i}.
 \end{equation}

Let us turn to Part~\ref{it:n3}. The relation in~\eqref{eq:p} when taken modulo 3 reads that
 $$
 \sum_{j=0}^{2^{k+1}} c_{k+1,j} x^j\equiv \sum_{i=0}^{2^k} c_{k,i} x^{2^{k+1}-2i}\pmod3.
$$
 Thus, $c_{k+1,j}\equiv c_{k,2^k-j/2}\pmod 3$ for all even $j$ between $0$ and $2^{k+1}$, while $c_{k+1,j}\equiv 0 \pmod 3$ for odd $j$ (in fact, $c_{k+1,j}=0$ for all odd $j$ since $p_{k+1}$ is an even function). In terms of the sequences $(b_{\ell,j})_{j=0}^{2^\ell}$, this relation states that
 $$
 b_{k+1,j}\equiv b_{k,j/2}\pmod 3 \quad \mbox{for all even $j$ with $0\le j\le 2^k$,}
$$
 while $b_{k+1,j}\equiv 0\pmod 3$  for all odd~$j$. This implies Part~\ref{it:s1}. 
For Part~\ref{it:s2}, the relation in~\eqref{eq:p} when taken modulo 9 gives that $c_{k+1,0}\equiv c_{k,2^k}$ and $c_{k+1,2^{k+1}}\equiv c_{k,0}\cdot 2^{2^k} + c_{k,1}\cdot3\cdot 2^{2^k-1}$. 
Since $c_{k,1}$ is divisible by $3$, we have in fact that $c_{k+1,2^{k+1}}\equiv c_{k,0}\cdot 2^{2^k} \equiv c_{k,0} \pmod 9$. By the induction hypothesis, this implies that $9$ does not divide $b_{k+1,0}$.

By the argument above, $c_{k+1,2^{k+1}}$ is non-zero module $3$ for odd $k$ and non-zero module 9 for even $k$. Thus, regardless of the parity of $k$, the degree of the polynomial $p_{k+1}$ is exactly $2^{k+1}$. Moreover, $p_{k+1}$ satisfies Eisenstein's criterion for prime $q=3$ (namely, that $q$ divides all coefficients, except exactly one at the highest power of $x$ or at the constant term while the other of the two is not divisible by $q^2$). By the criterion (whose proof can be found in e.g.\ \cite[Section~4]{PI14}), the polynomial $p_{k+1}$ is irreducible, proving Part~\ref{it:n4}.

By putting the above claims together, we see that $\mu_{k+1}$ is a root of an irreducible polynomial of degree $2^{k+1}$, establishing Part~\ref{it:n5}. This completes the proof the lemma (and thus of Theorem~\ref{THM:min-poly-degree})\end{proof}

\section{Concluding remarks}\label{SEC:remarks}
Our proof of Theorem~\ref{THM:min-poly-degree} shows that for every integer $d$ which is a power of $2$ there exists a finite family $\mathcal{F}$ of $r$-graphs such that $\pi(\mathcal{F})$ has algebraic degree $d$.
It seems interesting to know whether this is true for all positive integers.

\begin{problem}\label{PROB:degree-for-all-integers}
Let $r\ge 3$ be an integer.
Is it true that for every positive integer $d$ there exists a finite family $\mathcal{F}$ of $r$-graphs such that $\pi(\mathcal{F})$ has algebraic degree exactly~$d$?
\end{problem}

By considering other patterns, one can get Tur\'{a}n densities in $\Pi_{\mathrm{fin}}^{(r)}$ whose algebraic degrees are not powers of $2$. For example, the pattern $([3], \multiset{1,2,3}, \{1,2\})$ with \emph{recursive parts} 1 and 2 (where we can take blowups of the single edge $\multiset{1,2,3}$ and recursively repeat this step inside the first and the second parts of each added blowup)
gives a Tur\'{a}n density in $\Pi_{\mathrm{fin}}^{(3)}$ (by~\cite[Theorem~3]{PI14}, a generalisation of Theorem~\ref{THM:min-pattern}) whose degree can be computed to be~$3$. However, we did not see any promising way of how to produce a  pattern whose Lagrangian has any given degree~$d$.
\section*{Acknowledgement}
We would like to thank the referees for their helpful comments. 
\small
\bibliographystyle{abbrv}

\end{document}